\documentclass{article}[12pt]
\usepackage{amsmath}
\usepackage{amsfonts}
\usepackage{amssymb}
\usepackage{appendix}
\usepackage{amsthm }
\usepackage{footnote}
\usepackage{fullpage}

\usepackage[numbers,colon]{natbib}

\usepackage{enumerate}
\newtheorem{Thm}{Theorem}

\newtheorem{Pro}[Thm]{Proposition}

\makeatletter
\def\blfootnote{\xdef\@thefnmark{}\@footnotetext}
\makeatother

\theoremstyle{definition}

\theoremstyle{remark}
\newtheorem{Rem}[Thm]{\bf{Remark}}

\newcommand{\ConvFDD}{\overset{f.d.d.}{\longrightarrow}}
\newcommand{\EqFDD}{\overset{f.d.d.}{=}}

\newcommand{\Cov}{\mathrm{Cov}}
\newcommand{\Var}{\mathrm{Var}}
\newcommand{\E}{\mathbb{E}}

\title{Short-range dependent processes subordinated to the Gaussian may not be strong mixing}

\author{Shuyang Bai\qquad Murad S. Taqqu}

\begin{document}
\maketitle
\begin{abstract}
There are all kinds of weak dependence. For example, strong mixing. Short-range dependence (SRD) is also a form of weak dependence. It occurs in the context of processes that are subordinated to the Gaussian. Is a SRD process strong mixing if the underlying Gaussian process is long-range dependent? We show that this is not necessarily the case. 
\end{abstract}
\blfootnote{
\begin{flushleft}
\end{flushleft}
\textbf{2010 AMS Classification: 60G18} \\
}

\section*{}
Let $\{Z_i\}$ be a standardized Gaussian process with covariance function $\gamma(n)=n^{2H-2}L(n)$, where $1/2<H<1$ and $L(n)$ is slowly varying.  We will consider instantaneous transformations $X_i=P(Z_i)$, where $\E P(Z_i)^2<\infty$. The sequence $\{X_i\}$ is said to be LRD if the sum of its covariances diverges and SRD if the sum converges. Note that the sequence $\{Z_i\}$ is LRD because $\sum_{n=-\infty}^{+\infty}\gamma(n)=\infty$. The sequence $\{X_i\}$, however, may be LRD or SRD depending on $P(x)$.

Suppose now  that $P(\cdot)$ is a finite-order polynomial. It can then be expressed as 
\[
P(x)=c_0+\sum_{k=m}^n  c_k H_k(x),    \quad 1\le m\le n,  
\]
with $c_m\neq 0$, where $H_k(x)$ is the $k$-th order Hermite polynomial. The bottom index $m$ is called the \emph{Hermite rank} of $P(x)$ and/or of the process $\{P(X_i)\}$.

It is known from \citet{breuer:major:1983:central} that when 
\begin{equation}\label{eq:SRD rank}
(2H-2)m+1<0,
\end{equation}
which can only happen when $m\ge 2$, 
then $\{X_i\}$ is SRD and as $N\rightarrow\infty$, 
\[
N^{-1/2} \sum_{i=1}^{[Nt]} \left[ P(Z_i)-\E P(Z_i) \right] \ConvFDD \sigma B(t), 
\]
where $\sigma^2=\sum_n \gamma(n)$, $B(t)$ is the standard Brownian motion and $\ConvFDD$ denotes convergence of finite-dimensional distributions. This seems to suggest that $\{P(Z_i)\}$  has weak dependence. It is natural to ask whether $\{P(Z_i)\}$ is strong mixing.\footnote{
 A stationary process $\{X_i\}$ is said to be strong mixing if 
\[\lim_{k\rightarrow\infty}\sup\left\{|P(A)P(B)-P(A\cap B)|, ~A\in \mathcal{F}_{-\infty}^0,B\in \mathcal{F}_{k}^{\infty} \right\}=0,
\] 
where $\mathcal{F}_a^b$ is the $\sigma$-field generated by $X_a,\ldots, X_b$.}
We will show that this may \emph{not} be the case.
\begin{Thm}\label{Thm:contradict}
Suppose that $\{Z_i\}$ is LRD with covariance  $\gamma(n)=n^{2H-2}L(n)$, where $H$ satisfies (\ref{eq:SRD rank}). 
The SRD process $\{X_i=P(Z_i)\}$ is not strong mixing if there exists a polynomial $Q(x)$ such that the Hermite rank $m'$  of $Q(P(x))$  satisfies 
\begin{equation}\label{eq:change rank}
(2H-2)m'+1>0.
\end{equation}
\end{Thm}
\begin{Rem}
The process $\{X_i=P(Z_i)\}$ in the theorem is SRD. The theorem states that this process is not strong mixing if there a polynomial $Q(x)$ such that the new process $\{Q(P(Z_i))\}$ is LRD. Note that (\ref{eq:change rank}) implies, in view of (\ref{eq:SRD rank}), that $m'<m$.
\end{Rem}

\begin{proof}
We argue by contradiction. Suppose that $\{X_i\}$ is strong mixing. Then by the definition of strong mixing, 
$\{Q(X_i)\}$  is also strong mixing. But (\ref{eq:change rank}) implies that  (\citet{Taqqu:1975})
\begin{equation}\label{eq:contradiction}
s_N^2:=\Var\left[\sum_{i=1}^N Q(X_i)\right]\sim c_H L(N)^{m'}N^{(2H-2)m'+2}, \quad (2H-2)m'+2>1.
\end{equation}
On the other hand, $S_N:=\sum_{i=1}^N \left[Q(X_i)-\E Q(X_i)\right]$ is an element living on  Wiener chaos of a finite order (see \citet{janson:1997:gaussian}, Chapter 2). By \citet{janson:1997:gaussian}, Theorem 5.10, for any $p>2$, there exists a constant $c_p>0$ depending only on $p$, such that
\[
\E \left|s_N^{-1} S_N\right|^p \le c_p \left(\E \left|s_N^{-1} S_N\right|^2\right)^{p/2} =c_p.
\]
Therefore $\{s_N^{-2} S_N^2, N\ge 2\}$ is uniformly integrable. By Theorem 1.3 of \citet{peligrad:1986:recent}, strong mixing and uniform integrability imply that 
\[
s_N^2= l(N) N
\]
for some slowly varying function $l(N)$. This contradicts  (\ref{eq:contradiction}). 
\end{proof}

In some cases, no polynomial $Q(x)$ satisfies the requirement of Theorem \ref{Thm:contradict}. For example, when $P(x)=x^2$, then the Hermite rank $m=2$, and one always has
\[
\E Q(Z^2)H_1(Z)=\E Q(Z^2) Z=0
\]
for arbitrary polynomials $Q(x)$ (in fact for arbitrary $L^2(\Omega)$ functions). This is because $Q(Z^2)$ is an even function of $Z$.
So the Hermite rank of $Q(P(x))$ is  at least $2$, and hence we don't have $m'<m$.

In the simple case where $P(x)$ is a Hermite polynomial, we have the following result:
\begin{Pro}\label{Pro:H_m^p}
Suppose $P(x)=H_m$(x), $m\in \mathbb{Z}_+$. The polynomial $Q(x)$  required in Theorem \ref{Thm:contradict} exists in either of the following cases:
\begin{enumerate}[(a)]
\item $m\ge 4$ is even and $H>3/4$. 
\item $m\ge 3$ is odd.
\end{enumerate}
\end{Pro}
\begin{proof}
Using the product formula ((3.13) of \citet{janson:1997:gaussian}) for Hermite polynomial, one has
\begin{equation}\label{eq:H_m^2}
H_m(x)^2=\sum_{k=0}^m k! {m\choose k}^2 H_{2m-2k}(x),
\end{equation}
\begin{equation}\label{eq:H_m^3}
H_m(x)^3=\sum_{k_1=0}^m  \sum_{k_2=0}^{(2m-2k_1)\wedge m}  k_1!k_2! {m  \choose k_1}^2 {2m-2k_1 \choose k_2} {m\choose k_2} H_{3m-2k_1-2k_2}(x).
\end{equation}

For case (a), choose $3/4<H<1$, but not too big such that $\{P(X_i)=H_m(X_i)\}$ is SRD. This will happen by constraining $H$ to satisfy (\ref{eq:SRD rank}). Now choose $Q(x)=x^2$. Then by (\ref{eq:H_m^2}), 
\[
Q(P(x))=H_m(x)^2= m!+(m-1)!m^2 H_2(x)+\ldots,
\] 
so $\{Q(P(Z_i))\}$ has Hermite rank $m'=2$, which is less than $m\ge 4$. Since $m'=2$, and $H>3/4$, we conclude that $\{Q(P(Z_i))\}$ is LRD and satisfies (\ref{eq:change rank}).

For case (b), choose $Q(x)=x^3$. Then 
\[
Q(P(x))=H_m(x)^3=a_1H_1(x)+\ldots
\] 
for some $a_1>0$. The term $H_1(x)$ appears when $3m-2k_1-2k_2=1$, e.g., when $k_1=(m-1)/2$,  $k_2=m$. The coefficient $a_1>0$ because all the coefficients before the Hermite polynomials in (\ref{eq:H_m^3}) are positive.
It is then clear that the Hermite rank of $H_m(x)^3$ is  $m'=1$. Hence the polynomial $Q(x)$ satisfies (\ref{eq:change rank}).
\end{proof}
\begin{Rem}
In Proposition \ref{Pro:H_m^p} case (b), we do not need a restriction on $H$. We require $m\ge 3$ since $m=1$ is incompatible with (\ref{eq:SRD rank}). 
\end{Rem}

\begin{Rem}
What about the converse? Can a strong mixing process  not be subordinated to a Gaussian LRD process? The answer is clearly ``yes''. Suppose for example  $\{X_i\}$ i.i.d.\ Gaussian. Then  there is no $\{X_i'\}\EqFDD\{X_i\}$ so that $X_i'=G(Z_i')$, where $\{Z_i'\}$ is LRD Gaussian, because the covariance $\Cov[X_i',X_0']\neq 0$ for large $i$.
\end{Rem}

\medskip
\noindent\textbf{Acknowledgments.} This work was partially supported by the NSF grant  DMS-1309009 at Boston University.
\bibliographystyle{plainnat}
\bibliography{Bib}

\begin{thebibliography}{4}
\providecommand{\natexlab}[1]{#1}
\providecommand{\url}[1]{\texttt{#1}}
\expandafter\ifx\csname urlstyle\endcsname\relax
  \providecommand{\doi}[1]{doi: #1}\else
  \providecommand{\doi}{doi: \begingroup \urlstyle{rm}\Url}\fi

\bibitem[Breuer and Major(1983)]{breuer:major:1983:central}
P.~Breuer and P.~Major.
\newblock Central limit theorems for non-linear functionals of {G}aussian
  fields.
\newblock \emph{Journal of Multivariate Analysis}, 13\penalty0 (3):\penalty0
  425--441, 1983.

\bibitem[Janson(1997)]{janson:1997:gaussian}
S.~Janson.
\newblock \emph{Gaussian Hilbert Spaces}, volume 129.
\newblock Cambridge University Press, 1997.

\bibitem[Peligrad(1986)]{peligrad:1986:recent}
M.~Peligrad.
\newblock Recent advances in the central limit theorem and its weak invariance
  principle for mixing sequences of random variables (a survey).
\newblock In E.~Eberlein and M.~S. Taqqu, editors, \emph{Dependence in
  Probability and Statistics}, pages 193--223. Birkhauser, 1986.

\bibitem[Taqqu(1975)]{Taqqu:1975}
M.S. Taqqu.
\newblock Weak convergence to fractional {B}rownian motion and to the
  {R}osenblatt process.
\newblock \emph{Zeitschrift f\"{u}r Wahrscheinlichkeitstheorie und Verwandte
  Gebiete}, 31:\penalty0 287--302, 1975.

\end{thebibliography}

\medskip
\noindent Shuyang Bai~~ \textit{bsy9142@bu.edu}~~~~
Murad S. Taqqu ~~\textit{murad@bu.edu}~~~~\\
Department of Mathematics and Statistics\\
111 Cummington Mall\\
Boston, MA, 02215, US

\end{document}